\documentclass[12pt]{amsart}
\usepackage{epsfig,comment}
\usepackage{mathtools}
\usepackage[subnum]{cases}
\usepackage{enumerate}
\usepackage{bbm,bm}
\usepackage{amsmath,amssymb,mathrsfs,amsthm}
\usepackage{color}
\usepackage{ bbold }
\usepackage[bookmarks=true,
bookmarksnumbered=true, breaklinks=true,
pdfstartview=FitH, hyperfigures=false,
plainpages=false, naturalnames=true,
colorlinks=true,pagebackref=true,
pdfpagelabels,colorlinks=true,
	linkcolor=blue,
	citecolor=black,
	filecolor=black,
	urlcolor=blue]{hyperref}
\usepackage{wasysym}
\usepackage{tikz,tikz-3dplot}

\usepackage{pgfplots}
\pgfplotsset{width=10cm,compat=1.9}
\usepgfplotslibrary{fillbetween}

\theoremstyle{plain}
\newtheorem{theorem}{Theorem}[section]
\newtheorem{prop}[theorem]{Proposition}
\newtheorem{conjecture}[theorem]{Conjecture}
\newtheorem{corollary}[theorem]{Corollary}
\newtheorem{lemma}[theorem]{Lemma}
\theoremstyle{definition}

\theoremstyle{remark}
\newtheorem{remark}[theorem]{Remark}

\DeclareMathOperator{\conv}{conv}

\DeclareMathOperator{\vol}{vol}

\DeclareMathOperator{\tr}{tr}
\DeclareMathOperator{\Int}{int}

\DeclareMathOperator{\pyr}{pyr}
\DeclareMathOperator{\bipyr}{bipyr}

\title{A note concerning frames and geometric inequalities}
\author{J. Ledford, K. Rivera-Ayala, E. Schroeder}
\address{Longwood University, 201 High St., Farmville, VA, 23901. USA}
\email{ledfordjp@longwood.edu}
\email{kevin.riveraayala@live.longwood.edu}
\email{emma.schroeder@live.longwood.edu}
\date{}

\thanks{The authors would like to thank the Perspectives on Research In Science \& Mathematics program at Longwood University for sponsoring this research.}

\begin{document}
\begin{abstract}
   One may associate several frames to a given polytope, such as its collection of vertices, edges, or facet normal vectors.  In this note, we use these frames to generate geometric inequalities for the simplex in $\mathbb{R}^d$ and polytopes with $d+2$ vertices in dimension 2 and 3.
   \end{abstract}


\maketitle

\section{Introduction}

Our goal in this paper is to prove geometric inequalities for polytopes using tools from frame theory.  The machinery of frames allows us to uncover some well-known as well as novel results. 

Geometric inequalities have been well studied, we seek inequalities which allow us to compare internal content with external content.  The most famous of these results is the isoperimetric inequality, that bounds the area $A$ inside a given curve by the square of its arclength $L$:
\[
4\pi A\leq L^2.
\]
Since we restrict our attention to polytopes, our results will look more like the well-known simplex version of the above inequality, \cite{veljan_1995_sine}.  For a simplex $T\subset\mathbb{R}^d$:
\[
\frac{\sqrt{d^{3d}(d+1)^{d+1}}}{d!}\vol_{d}(T)^{d-1}\leq\left(\sum_{F\in\mathcal{F}_{d-1}(T)} \vol_{d-1}(F)\right)^d.
\]
More recently, this theme has appeared in 
\cite{bilyk_2023} and \cite{cahill_2023}, with the focus on frame attributes that lead to optimal results.  In general, the results correspond to choosing the vertex vectors of a polytope as the frame vectors.

In this paper, we fix a polytope, then choose different sets of vectors for the frame vectors, such as the edges and normal vectors in addition to the vertices.  This allows us to generate several geometric inequalities for the simplex.

The remainder of this paper is organized as follows.  Section 2 contains definitions and basic facts that will be used throughout.  Section 3 provides results for vertex and edge frames associated to a simplex.  Section 4 presents results for the normal frame of a simplex and includes a version of the isoperimetric inequality.  In Section 5, we collect results for polytopes with $d+2$ vertices, explicitly examining the low-dimensional cases while providing a conjecture for higher dimensions.

%%%%%%%%%%%%%%%%
%%%%%%%%%%%%%%%% Preliminaries
%%%%%%%%%%%%%%%%

\section{Preliminaries}

In this section, we collect well-known results on finite frames and geometry that will be useful in the sequel.  For the set of real numbers, we use $\mathbb{R}$, while $\mathbb{R}^d$ is the collection of $d$-vectors.  We use $\langle\cdot,\cdot \rangle$ for the standard inner product.  For $p\in [1,\infty)$, by $\ell_p$, we mean the collection of sequences $y=(y_j)$ such that $\| y\|_{\ell_p}:=\left(\sum|y_j|^p\right)^{\frac{1}{p}} <\infty $.

It is well known that a spanning set of vectors $\{v_1,\dots, v_n\}\subset\mathbb{R}^d$ forms a \emph{frame} for $\mathbb{R}^d$, \cite{Casazza_book}.  This means that there are positive numbers $A$ and $B$ such that for all $x\in\mathbb{R}^d$
\[
A\| x\|_{\ell_2}^2\leq \sum_{j=1}^{n} |\langle x, v_j \rangle|^2 \leq B \| x \|_{\ell_2}^{2}.
\]
Associated to a frame $\{v_1,\dots,v_n \}$ are three operators, the \emph{analysis operator} $F:\mathbb{R}^d \to \ell_2$ given by
\[
Fx:=\left( \langle x,v_j\rangle: 1\leq j\leq n \right),
\]
the \emph{synthesis operator} given by the adjoint of $F$, that is, $F^*:\ell_{2}\to \mathbb{R}^d$, where 
\[
F^*y=\sum_{j=1}^n y_j v_j,
\]
and the \emph{frame operator} $S:\mathbb{R}^d\to\mathbb{R}^d$ given by $S:=F^*F$.  
\begin{remark}
In terms of our frame vectors $\{v_1,\dots,v_n\}$, the synthesis operator is given by the $d\times n$ matrix whose columns vectors are $v_j$, $1\leq j\leq n$, while the analysis operator is the conjugate transpose of this matrix.  Hence the frame operator may be expressed as a $d\times d$ matrix.
\end{remark}

Note that we may write out the frame inequality as 
\[
A\| x \|_{\ell_2}^{2}\leq \langle Sx,x\rangle \leq B \| x\|_{\ell_2}^{2},
\]
hence any eigenvalue $\lambda$ of $S$, satisfies $0<A\leq \lambda \leq B$.  This implies that the frame operator $S$ is invertible.  This is called the \emph{Rayleigh-Ritz characterization} in \cite[Theorem 3.2]{rivin_2003}, but is implied by the \emph{Courant minmax principle} found in \cite[p. 31-34]{Courant_Hilbert_1953}.

Furthermore, we may apply the arithmetic mean-geometric mean inequality to the eigenvalues to deduce the \emph{trace-determinant inequality}:
\begin{equation}\label{T-D ineq}
    \det(S)\leq\left(\frac{\tr(S)}{d} \right)^d.
\end{equation}
This inequality may be exploited to produce several geometric inequalities.

In order to facilitate the use of this inequality, we use the following version of the Cauchy-Binet formula.

\begin{lemma}[Cauchy-Binet formula]\label{det S lemma}
Let $d\in\mathbb{N}$ and $n\geq d$ and suppose that $\mathcal{V}=\{v_1,v_2,\dots,v_n\}$ is a frame for $\mathbb{R}^d$ with frame operator $S$.  We have
\[
\det(S)=\sum_{j=1}^{\binom{n}{d}}\det(V_j^T V_j)=\sum_{j=1}^{\binom{n}{d}}\det(V_j)^2,
\]
where $V_j$ is a $d\times d$ matrix whose $d$ column vectors are selected from $\mathcal{V}$.  
\end{lemma}

We denote by $\mathcal{P}_{d}$ the collection of convex polytopes inscribed in the unit sphere $\mathbb{S}^{d-1}\subset\mathbb{R}^d$ that contain the origin in their interior and whose $d$-dimensional volume $\vol_d$ is non-zero.  For $P\in \mathcal{P}_d$ and $1\leq j\leq d$, we denote by $\mathcal{F}_j(P)$ the collection of $j$-dimensional faces of $P$.  The following sets of vectors are natural choices for frames (for $\mathbb{R}^d$):
\begin{align*}
    \text{the vertices } V(P)&=\mathcal{F}_{0}(P),\\
    \text{the edges } E(P)&=\mathcal{F}_1(P),\\
    \text{the centroids of the facets }C(P)&=\left\{c_F\in\mathbb{R}^d: c_F \text{ is centroid of } F\in\mathcal{F}_{d-1}(P) \right\},\text{ and}\\
    \text{ the unit normals }N(P)&=\{\hat{n}\in\mathbb{R}^d: n\perp F,F\in\mathcal{F}_{d-1}(P)\}.
\end{align*}
We denote the corresponding frame operators by $S_V(P), S_E(P),S_C(P)$ and $S_N(P)$, respectively.

\begin{remark}
    Since the convex hull of the centroids of the facets of $T$ is itself a simplex similar to $T$, applying the techniques to the centroid frame $C(T)$ reproduces Proposition \ref{thm: S_V simplex}. 
\end{remark}

\begin{remark}
    If $P\in\mathcal{P}_d$, then we could generate a frame for $\mathbb{R}^d $ as follows. For fixed $1\leq j\leq d$, associate to each $F\in\mathcal{F}_j(P)$ the vector $v_F$ that connects the origin to the $j$-dimensional centroid of $F$.  The collection $\{\hat{v}_F: F\in\mathcal{F}_j(P)\}$ is a frame for $\mathbb{R}^d$.  Scaling these vectors with positive weights also generates a frame for $\mathbb{R}^d$.

    In the construction above, we could replace the centroid with any point $p\in \Int(F)$, where $F\in\mathcal{F}_j(P)$. The analysis of such frames will be the subject of a subsequent project.
\end{remark}

%%%%%%%%%%%%%%%%%%%%%%
%%%%%%%%%%%%%%%%%%%%%% Inequalities for the vertex and edge frames
%%%%%%%%%%%%%%%%%%%%%%

\section{Vertex and Edge Frame Inequalities for the Simplex}

In this section, we consider a fixed but otherwise arbitrary simplex $T\in\mathcal{P}_d$ and apply the trace determinant inequality to the vertex and edge frame operators.  We begin with $S_V(T)$.
\begin{prop}\label{thm: S_V simplex}
    Suppose that $T\in\mathcal{P}_d$ is a simplex, then
    \begin{equation}\label{S_V geometric ineq}
    \|\vol_d(T)  \|_{\ell_2}\leq\frac{1}{d!}\left(1+\frac{1}{d}\right)^{\frac{d}{2}}.
\end{equation}
\end{prop}

\begin{proof}
For a simplex $T\in\mathcal{P}_d$, denote its vertex frame $V=\{v_1,\dots,v_{d+1}.\}$  For the frame operator $S_V(T)$, we have $\tr(S_V(T))=d+1$ and by Lemma \ref{det S lemma}, 
\[\det(S_V(T))=\sum_{j=1}^{d+1}\left(d!\vol_d(T_j)\right)^2=(d!)^2\|\vol_d(T) \|_{\ell_2}^2,\]
where $T_j$ is the simplex formed by vertex vectors from $T$, replacing the vertex $v_j$ with the origin.  Now \eqref{T-D ineq} provides the result.
\end{proof}

Noting that $\vol_{d}(T)=\|\vol_d(T)\|_{\ell_1}$ we may recover \cite[Theorem 2.2]{gerber_1975} by applying the Cauchy-Schwarz inequality.

\begin{corollary}\label{cor: vol bound}
    Suppose that $T\in\mathcal{P}_d$ is a simplex, then
    \begin{equation}\label{S_V l_1 ineq}
        \vol_{d}(T)\leq \frac{\sqrt{d+1}}{d!}\left( 1+\frac{1}{d}\right)^{\frac{d}{2}},
    \end{equation}
    with equality if and only if $T=\Delta$, the regular simplex.
\end{corollary}

Applying the trace-determinant inequality to $S_E(T)$, allows us to extend a result for the tetrahedron found in \cite{Scott_2005}.
\begin{prop}\label{thm: S_E simplex}
     Suppose that $T\in\mathcal{P}_d$ is a simplex, then
     \begin{equation}\label{S_E geometric ineq}
    \vol_d(T)\leq \frac{\|\vol_1(T) \|_{\ell_2}^d}{d!(d+1)^{\frac{d-1}{2}}d^{\frac{d}{2}}}.
\end{equation}
\end{prop}

\begin{proof}
    For $T\in\mathcal{P}_d$, consider its edge frame $E=\left\{v_1,\dots,v_{\binom{d+1}{2}}\right\}$, with frame operator $S_{E}(T)$. We have $\tr(S_E(T))=\|\vol_1(T) \|_{\ell_2}^{2}$ by definition, and Lemma \ref{det S lemma} provides $\det(S_E(T))=c_d(d!\vol_d(T))^2$, where $c_d$ is the number of spanning sets of $\mathbb{R}^d$ formed by the $\binom{d+1}{2}$ edge vectors of $T$.  The formula $c_d=(d+1)^{d-1}$ for the coefficient is due to Cayley \cite{cayley_1878}.
\end{proof}

\begin{remark}
    In both Propositions \ref{thm: S_V simplex} and \ref{thm: S_E simplex}, the equality condition corresponds to the simplex $T$ producing a tight frame.  Perhaps unsurprisingly, for the regular simplex $\Delta$, $S_E(\Delta)$ and $S_V(\Delta)$ are tight frames.  The result for the vertex frame may be found in \cite{Fickus_Schmitt_2020}.  To see that the edge frame is tight, we note that Theorem \ref{thm: S_E simplex} follows from the Veljan-Korchm\'{a}ros inequality, \cite[Corollary 1]{veljan_1995_sine}
    \[
    \vol_d(T)\leq \frac{1}{d!}\sqrt{\frac{d+1}{2^d}}\prod_{e\in\mathcal{F}_1(T)}\vol_{1}(e)^{\frac{2}{d+1}},
    \]
    and applying the arithmetic mean geometric mean inequality to the right-hand side.  For both inequalities, equality occurs if and only if the simplex is regular.
\end{remark}

Combining Propositions \ref{thm: S_V simplex} and \ref{thm: S_V simplex} and applying H\"{o}lder's inequality yields the following theorem.

\begin{theorem}\label{thm: l_p}
    Suppose that $T\in\mathcal{P}_d$ is a simplex and $1\leq p \leq 2$, then
    \[
\|\vol_d(T) \|_{\ell_p}\leq \frac{(d+1)^{\frac{3d+1}{2}-\frac{2d}{p}}}{d!d^{\frac{d}{2}}}\|\vol_1(T) \|_{\ell_2}^{\frac{2}{p}-1},
\]
with equality if and only if $T=\Delta$.
\end{theorem}

%%%%%%%%%%%%
%%%%%%%%%%%% Normal Frame Inequalities for the Simplex
%%%%%%%%%%%%
\section{Normal Frame Inequalities for the Simplex}

Let $T\in\mathcal{P}_d$, for $F\in\mathcal{F}_{d-1}(T)$, we let $n_F$ denote the outward normal vector of $F$ such that $\|n_F \|_{\ell_2}=\vol_{d-1}(F)$ and $\hat{n}_F$ the corresponding unit vector.  In this section, we consider the following frames for $\mathbb{R}^d$
\begin{align*}
    \widehat{N}(T)&=\{\widehat{n}_F: F\in\mathcal{F}_{d-1}(T)\}\quad\text{ and }\\
    N(T)&=\{\vol_{d-1}(F)\widehat{n}_F: F\in\mathcal{F}_{d-1}(T)\}.
\end{align*}
The convex hull of the corresponding vectors is called the \emph{unit normal simplex} and the \emph{normal simplex}, respectively.

\begin{comment}

Applying \eqref{T-D ineq} to these frames yields the following.
\begin{theorem}\label{thm: N1 ineq}
Let $T\in\mathcal{P}_d$ and $p\in \Int(T)$.  Then we have the weighted norm inequality
\[
\| \vol_{d}(\widehat{N}_T) \|_{\ell_2,w(p)}\leq \dfrac{1}{d!}\left(\frac{1}{d}\sum_{j=1}^{d+1}h_j^2\right)^{\frac{d}{2}},
\]
where the weights are given by $w_j(p)=\displaystyle\prod_{i\neq j} h_i^2$.
\end{theorem}
\begin{proof}
Apply the argument given in Theorem \ref{thm: S_V simplex} to $N_1(T,p)$.
\end{proof}
\begin{remark}
If we take $p=i_T$, the incenter of $T$, we recover Theorem \ref{thm: S_V simplex} for $\widehat{N}_T$.
\end{remark}

\begin{remark}
    If we take $p=0$, then we may write
    \[
    \|\vol_{d}(N_T) \|_{\ell_2} \leq \dfrac{1}{d!}\left(\frac{1}{d} \sum_{j=1}^{d+1}h_j^2\right)^{\frac{d}{2}}=\dfrac{d^{\frac{d}{2}}}{d!}\|\vol_d(T) \|_{\ell_2,w}^d,
    \]
    where $N_T=\conv(N_1(T,0))\subset T$ and the weights are given by $w_j=\vol_{d-1}(F_j)^{-2}$.
\end{remark}

\begin{theorem}\label{thm: N2 ineq}
    For $T\in\mathcal{P}_d$, let $N'_T=\conv(N_2(T))$, then we have
    \[
    \|\vol_d(N_T') \|_{\ell_2} \leq \frac{1}{d!d^{\frac{d}{2}}}\| \vol_{d-1}(T) \|_{\ell_2}^{d}.
    \]
\end{theorem}
\begin{proof}
    Apply the argument given in Theorem \ref{thm: S_V simplex} to $N_2(T)$.
\end{proof}

\end{comment}

We begin with a local result.  For a fixed vertex $v\in T$ and $1\leq j\leq d$, define
\[
\mathcal{F}_{j}(T,v):=\left\{ F\in\mathcal{F}_j(T): v\in F \right\}.
\]

\begin{lemma}\label{lem: local vol}
    Suppose $T\in\mathcal{P}_d$ is a simplex.  For all vertices $v$, we have 
    \[
    d!\vol_{d}\left(\conv(n_F:F\in\mathcal{F}_{d-1}(T,v))\right)=(d!\vol_{d}\left(\conv(\mathcal{F}_{1}(T,v))\right))^{d-1}=(d!\vol_d(T))^{d-1}.
    \]
\end{lemma}

\begin{proof}
The argument for $\mathbb{R}^2$ relies on elementary trigonometry.  Suppose there is an angle $\theta$ at some vertex.  Since the sum of the internal angles of a quadrilateral is $2\pi, $ moving the normal vectors into standard position yields a central angle of $\pi-\theta$, as shown below.

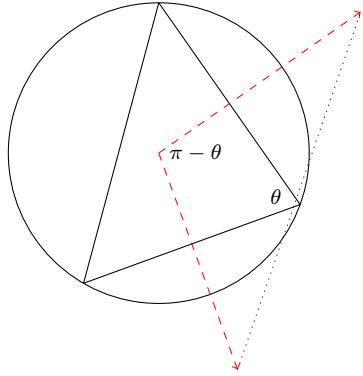
\begin{center}
    \begin{tikzpicture}[scale=2]
        % Draw the scaled unit circle with radius 2
        \draw[black] (0,0) circle (1);
        
        % Define triangle vertices
        \coordinate (A) at (0,1);
        \coordinate (B) at ({cos(240)},{sin(240)});
        \coordinate (C) at ({cos(340)},{sin(340)});

        %draw triangle
        \draw[black] (A) -- (B);
        \draw[black] (B) -- (C);
        \draw[black] (C) -- (A);

        % normal vectors
        \coordinate (n1) at ({sin(240)-1},{-cos(240)});
        \coordinate (n2) at ({-sin(240)+sin(340)},{cos(240)-cos(340)});
        \coordinate (n3) at ({1-sin(340)},{cos(340)});

        % \draw[red,dashed,->] (0,0) -- (n1);
        \draw[red,dashed,->] (0,0) -- (n2);
        \draw[red,dashed,->] (0,0) -- (n3);
        \draw[black,dotted] (n2) -- (n3);

        %labels
        \draw ($(C)-(.05,-.05)$) node[left] {\tiny$\theta$};
        \draw ($(0,0) $) node[right] {\tiny$\pi-\theta$};

    \end{tikzpicture}
 
\end{center}
If we denote the normal vector that intersects $e_j$ by $n_j$, then $\|n_j\|=\|e_j\|$.  The area of the \emph{normal triangle} is given by $A=\frac{1}{2}\| n_1\| \| n_2\|\sin(\pi-\theta)=\frac{1}{2}\| e_1\| \| e_2\|\sin\theta$, where the edges $e_1$ and $e_2$ meet at the chosen vertex.

\begin{comment}

In $\mathbb{R}^3$, the result follows from the scalar triple product.  We use the wedge notation for the cross product to set up the high  Suppose that the tetrahedron is given by $T=\conv(v_1,\dots,v_4)$ and let $e_{ij}$ denotes the edge connecting $v_i$ to $v_j$.  Let $n_1=e_{12}\wedge e_{13}, n_2=e_{13}\wedge e_{14},$ and $n_3=e_{14}\wedge e_{12}$ be the outward facing normals of the facets containing $v_1$.  Then we have
\begin{align*}
    \begin{vmatrix}
        n_1&n_2&n_3
    \end{vmatrix}&=|n_1\wedge n_2\wedge n_3|\\
    %&=|n_1\wedge(n_2\wedge n_3)|\\
    &=|(e_{12}\wedge e_{13})\cdot(n_2\wedge n_3)|\\
    &=|(e_{12}\cdot n_2)(e_{13}\cdot n_3)-(e_{12}\cdot n_3)(e_{13}\cdot n_2)|\\
    &=|(e_{12}\cdot(e_{13}\wedge e_{14}))(e_{13}\cdot(e_{14}\wedge e_{12}))-0|\\
    &=\begin{vmatrix}
        e_{12}& e_{13}&e_{14}
    \end{vmatrix}^2
\end{align*}

\end{comment}

For higher dimensions, denote by $F_j$ the facet not containing vertex $v_j$, whose corresponding outward normal vector is $n_j$.  If we restrict our attention to a specific vertex, say $v_i$, then the facets $F_j$, where $j\neq i$ contain $v_i$, hence the corresponding normal frame consists of the vectors $n_j=\star(\bigwedge_{k\neq j}e_{i,k})$, where $\star$ denotes the Hodge star.  For a fixed $i$, let $N_i(T)$ denote the matrix whose rows are given by $n_j$ and let $E$ denote the matrix whose columns are given by $e_{i,j}$.  Then $N_i(T)E=cI$, where the diagonal entries are $c=(d-1)!\vol_{d-1}(\conv(\{e_{i,k}:k\neq j\}))e_{i,j}\cdot n_j=d!\vol_d(T)$, where the second equality comes from the cone volume formula.  Hence we have $\det(N_T)=(d!\vol_d(T))^{d-1}$.  
\end{proof}

\begin{comment}\begin{align*}
    |n_1\cdot(n_2\times n_3)|&=|(e_{12}\times e_{13})\cdot(n_2\times n_3)|\\
    &=|(e_{12}\cdot n_2)(e_{13}\cdot n_3)-(e_{12}\cdot n_3)(e_{13}\cdot n_2)|\\
    &=|(e_{12}\cdot (e_{13}\times e_{14}))(e_{13}\cdot (e_{14} \times e_{12}))-0|\\
    &=|e_{12}\cdot(e_{13}\times e_{14})|^2.
\end{align*}
\end{comment}

\begin{corollary}\label{cor: vol N_T}
    Suppose that $T\in \mathcal{P}_d$ is a simplex.  For the normal simplex $N_T$, we have  
    \[
    \vol_{d}(N_T)=(d+1)d!^{d-2}\vol_d(T)^{d-1}.
    \]
\end{corollary}

Applying the trace determinant inequality to $S_V(N_T)$ yields the following.
\begin{prop}\label{prop: N_T ineq}
Suppose that $T\in\mathcal{P}_d$ is a simplex.  For the normal simplex, we have
\[
\| \vol_{d}(N_T) \|_{\ell_2}\leq \frac{1}{d!d^{\frac{d}{2}}}\| \vol_{d-1}(T)\|_{\ell_2}^{d}. 
\]
\end{prop}

This allows us to prove a version of the isoperimetric inequality.

\begin{theorem}\label{thm: simplex iso ineq}
    Suppose that $T\in\mathcal{P}_d$, then
    \[
     \sqrt{d+1}d!^{d-1}d^{\frac{d}{2}}  \leq \frac{\|\vol_{d-1}(T) \|_{\ell_2}^{d}}{\vol_d(T)^{d-1}}.
    \]
\end{theorem}

\begin{proof}
    Combine Proposition Corollary \ref{cor: vol N_T} with Proposition \ref{prop: N_T ineq}.
\end{proof}

\begin{comment}

We could iterate Theorem \ref{thm: simplex iso ineq} to get lower dimensional simplices.  In $\mathbb{R}^3$, we have the following result.

\begin{corollary}\label{cor: R3 iso ineq}
    Suppose that $T\in\mathcal{P}_3$, then
    \[
    \vol_{3}(T)\leq \frac{\sqrt[4]{2}}{144\sqrt{6}}\left\langle \mathbf{e}   , A\mathbf{e}   \right\rangle^{\frac{3}{4}},
    \]
    where $\mathbf{e}$ is the vector of squared edge lengths and 
    \[
    A=\begin{bmatrix}
        2 & 1 & 1 & 1 & 1 & 0\\
        1 & 2 & 1 & 1 & 0 & 1\\
        1 & 1 & 2 & 0 & 1 & 1\\
        1 & 1 & 0 & 2 & 1 & 1\\
        1 & 0 & 1 & 1 & 2 & 1\\
        0 & 1 & 1 & 1 & 1 & 2 
    \end{bmatrix}.
    \]
\end{corollary}
\end{comment}

\begin{comment}
\begin{corollary}
    Suppose that $Q\in \mathcal{P}_d$ with $n$ vertices and let $N_Q$ be the polytope whose vertices are given by the outward surface normals of $Q$.  Then  
    \[
    \vol_{d}(N_Q)=(n+1)\vol_d(Q).
    \]
\end{corollary}
\end{comment}

Applying Proposition \ref{thm: S_E simplex} to $\widehat{N}_T$ yields the following result.
\begin{prop}
    Suppose $T\in\mathcal{P}_d$, then
    \[
    \vol_d(\widehat{N}_T)\leq \dfrac{\|\vol_1(\widehat{N}_T)\|_{\ell_2}^{d}}{d!(d+1)^\frac{d-1}{2}d^{\frac{d}{2}}}.
    \]
\end{prop}

Shifting our focus to the unit normal frame, we may compute the volume of the (unit) normal cone along the edge $e$, denoted $NC(e,T)$.  Suppose that $e$ is incident to the $d-1$ facets $F_{e,1},\dots,F_{e,{d-1}}$ and denote the corresponding normal vectors by $n_{j}$, we adjust Lemma \ref{lem: local vol} by replacing a unit normal vector with the edge $e$ and computing the determinant, this yields
\[
\vol_{d-1}(NC(e_,T))=\|\star(n_{1}\wedge\cdots\wedge n_{{d-1}}) \|= \dfrac{(d!\vol_d(T))^{d-2}}{(d-1)!}\| e \|.
\]
\begin{remark}
In certain cases, we can use the method laid out in \cite{ribando_2006} to calculate the solid angle of the normal cone.  This would allow us to compute the external angle $\gamma(e)$.
\end{remark}

\begin{remark}
Concerning the external angle, in $\mathbb{R}^3$, we may use the techniques developed in \cite{bilyk_2023} (which extend those in \cite{cahill_2023}) applied to the unit normal frame to prove an analog of Corollary 3.3 found there
\[
    \Gamma_s(T):=\sum_{e\in\mathcal{F}_1(T)}\gamma(e)^s\leq \sum_{e\in \mathcal{F}_1(\triangle)}\gamma(e)^s
\]
for all $0<s\leq 1$. 
\end{remark}

\section{Polytopes with $d+2$ vertices}
Moving one vertex up from simplices, we have polytopes with $d+2$ vertices in $\mathbb{R}^d$.  Our strategy will be to decompose these polytopes into the union of simplices and use Lemma \ref{det S lemma} to account for duplication.  This means in general that our $\|\vol_d(Q) \|_{\ell_2}=\|\vol_d(Q) \|_{\ell_2,p}$, where $p\in Q$ is chosen appropriately.  We will prove results analogous to Theorems \ref{thm: S_V simplex} and \ref{thm: S_E simplex}.  In order to accomplish this, we complete the edge graph of the polytope by adding appropriate edges.  We denote this added content by $Q^o$, thus $\| \vol_1(Q^o)\|_{\ell_2}^{2}$ gives the sum of the squares of the lengths of the added edges.  In order to illustrate our methods, we will begin with $\mathbb{R}^2$ and $\mathbb{R}^3$.

%%%%%%%%%%%
%%%%%%%%%%% quadrilaterals
%%%%%%%%%%%
\subsection{Quadrilaterals}

In $\mathbb{R}^2$, the polytopes in question are quadrilaterals. We will consider the frame $S_E(Q\cup Q^o)$ generated by the edge vectors of $Q$, $\{e_j:1\leq j \leq 4\}$, and the two diagonals, $\{e_i^o,e_{ii}^o\}$.  The intersection of the diagonals is a natural choice for $p$.  For a quadrilateral $Q\in\mathcal{P}_2$, we label the vertices $v_1,\dots,v_4$ the edges $e_1,\dots,e_4$, and the regions bounded by the diagonals and a particular edge are labeled $T_1,\dots,T_4$, as shown in the figure below.

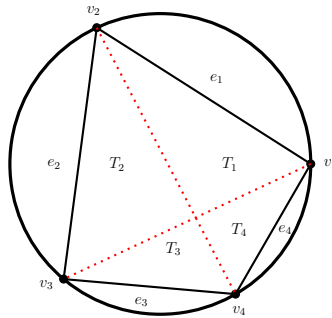
\begin{center}
     \begin{tikzpicture}[every node/.append style={transform shape}, scale=0.5]]
            \draw[very thick](0,0) circle (4);
        %vertices
            \filldraw ({4*cos(0)},{4*sin(0)}) circle (3pt);
            \filldraw ({4*cos(115)},{4*sin(115)}) circle (3pt);
            \filldraw ({4*cos(130)},{-4*sin(130)}) circle (3pt);
            \filldraw ({4*cos(300)},{4*sin(300)}) circle (3pt);
            \draw ({4.2*cos(0)},{4.2*sin(0)}) node[right] {$v_1$};
            \draw ({4.2*cos(115)},{4.2*sin(115)}) node[above] {$v_2$};
            \draw ({4.2*cos(130)},{-4.2*sin(130)}) node[left] {$v_3$};
            \draw ({4.2*cos(300)},{4.2*sin(300)}) node[below] {$v_4$};
        % edges
            \draw[thick] ({4*cos(0)},{4*sin(0)}) -- ({4*cos(115)},{4*sin(115)})-- ({4*cos(130)},{-4*sin(130)})-- ({4*cos(300)},{4*sin(300)})--({4*cos(0)},{4*sin(0)});
            \draw (1.5,2) node[above] {$e_1$};
            \draw (-2.5,0) node[left] {$e_2$};
            \draw (-.5,-3.4) node[below] {$e_3$};
            \draw (3,-1.75) node[right] {$e_4$};
        %diagonals
            \draw[dotted,red,thick] ({4*cos(0)},{4*sin(0)})--({4*cos(130)},{-4*sin(130)});
            \draw[dotted, red,thick] ({4*cos(115)},{4*sin(115)}) -- ({4*cos(300)},{4*sin(300)});
        %area regions
            \draw (1.5,0) node[right] {$T_1$};
            \draw (-1.5,0) node[right] {$T_2$};
            \draw (0,-2.25) node[right] {$T_3$};
            \draw (1.75,-1.75) node[right] {$T_4$};
    \end{tikzpicture}
\end{center}

We have
\begin{align*}
    \det(S_E(Q\cup Q^o)) & = 4\sum_{j\text{ (mod }4)=1}^{4}\det([e_je_{j+1}])^2 \\
    &=16\sum_{j \text{ (mod } 4)=1}^{4}\left(\vol_{2}(T_j)+\vol_{2}(T_{j+1})\right)^2 \\
    &= 16\left\langle  \begin{bmatrix}
    2&1&0&1\\
    1&2&1&0\\
    0&1&2&1\\
    1&0&1&2
\end{bmatrix} \begin{bmatrix}
    T_1\\T_2\\T_3\\T_4
\end{bmatrix},\begin{bmatrix}
    T_1\\T_2\\T_3\\T_4
\end{bmatrix}\right\rangle\\
&=:16\mathcal{Z}(Q).
\end{align*}

Now \eqref{T-D ineq} yields
\begin{equation}\label{eqn: quad Z}
\mathcal{Z}(Q)\leq \frac{\left(\|\vol_1(Q) \|_{\ell_2}^{2}+ \| \vol_1(Q^o)\|_{\ell_2}^{2}\right)^2}{64}.
\end{equation}

A slight rephrasing provides the following.
\begin{theorem}\label{thm: quad S_v ineq}
Suppose that $Q\in \mathcal{P}_2$ is a quadrilateral.  Taking $p$ to be the intersection of the diagonals, we have
\begin{equation}\label{thm: quad ineq 1}
\|\vol_{2}(Q) \|_{\ell_2,0}\leq 1,\text{ and}
\end{equation}
\begin{equation} \label{thm: quad ineq 2}
\displaystyle \|\vol_2(Q) \|_{\ell_2,p}^{2}+\prod_{j=1}^{2}\left(\vol_2(T_j)+\vol_2(T_{j+2})\right)\leq \dfrac{\left(\|\vol_1(Q) \|_{\ell_2}^2+\|\vol_1(Q^o)\|_{\ell_2}^2 \right)^2}{128}.\end{equation}
The first inequality becomes an equality if $Q$ is a square, while the second inequality is an equality if and only if $Q$ is a rectangle.    
\end{theorem}

\begin{proof}
The first inequality is \eqref{T-D ineq} applied to $S_V(Q)$,  the determinant side of the inequality has 6 terms, two of which are 0 when the vertices correspond to diagonals of a rectangle.  The second inequality is \eqref{eqn: quad Z}, noting that in the equality case, the areas of the four triangles formed by the vertex vectors and the edges are the same.  
\end{proof}

\begin{corollary}
    Suppose that $Q\in\mathcal{P}_2$ is a quadrilateral.  Let $\square\in\mathcal{P}_2$ be the square, then
    \[
    \vol_{2}(Q)\leq \vol_2(\square).
    \]
\end{corollary}
\begin{proof}
    Applying the Cauchy-Schwarz inequality and \eqref{thm: quad ineq 1}, we have
    \[
    \vol_2(Q)\leq 2\|\vol_2(Q) \|_{\ell_2,0} \leq 2
    \]
    The equality condition forces all $T_j$ to be the same, consequently, $Q$ must be a square.
\end{proof}

\begin{remark}
    The product in \eqref{thm: quad ineq 2} satisfies
    \[
    0\leq \prod_{j=1}^{2}\left(\vol_2(T_j)+\vol_2(T_{j+2})\right)\leq \frac{(\vol_2(Q))^2}{4}.
    \]
\end{remark}

In $\mathbb{R}^3$, the situation is a bit more complicated since there are exactly two combinatorial types of convex polytopes with 5 vertices, the pyramid with quadrilateral base and the triangular bipyramid, \cite{Britton_1973}.

%%%%%%%%%%%
%%%%%%%%%%% pyramids
%%%%%%%%%%%

\subsection{Pyramids}

For the pyramid $Q$, we introduce two edges one between $v_1$ and $v_3$ and the other between $v_2$ and $v_4$.

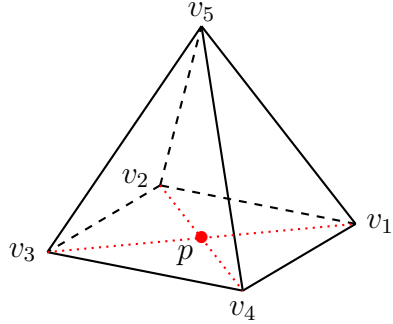
\begin{center}
 \tdplotsetmaincoords{70}{120} 
    
    \begin{tikzpicture}[tdplot_main_coords, line join=bevel]
    
      \coordinate (A) at (0,0,0);
      \coordinate (B) at (3,0,0);
      \coordinate (C) at (3,3,0);
      \coordinate (D) at (0,3,0);
      \coordinate (O) at (1.5,1.5,3);
    
        \draw[dashed,thick] (A) -- (B);
        \draw[thick] (B) -- (C);
        \draw[thick] (C) -- (D);
        \draw[dashed,thick] (D) -- (A);
    
        \draw[dashed,thick] (A) -- (O);
        \draw[thick] (B) -- (O);
        \draw[thick] (C) -- (O);
        \draw[thick] (D) -- (O);

        \draw[dotted,red,thick] (A)--(C);
        \draw[dotted,red,thick] (B)--(D);

        \node[left] at ($(A)+(0,0,0.1)$) {${v_2}$};
        \node[left] at (B) {${v_3}$};
        \node[below] at (C) {${v_4}$};
        \node[right] at (D) {${v_1}$};
        \node at ($(O)+(0,0,0.2)$) {${v_5}$};
        \node[red] at (1.5,1.5,0) {\textbullet};
        \node[below] at (1.5,1.25,0) {$p$};
        
    \end{tikzpicture}
\end{center}
Partitioning the quadrilateral base into 4 regions as before, including $v_5$ provides a partitioning of the volume of $Q$.  Denoting by $V_j$ the matrix whose columns are the edge vectors containing $v_j$, we have
\begin{align*}
    \det(S_E(Q\cup Q^o)) &=25\sum_{j=1}^{4} \det(V_jV_j^T)^2\\
    &=5^2 (3!)^2 \sum_{j\text{ (mod }4)=1}^{4}\left( \vol_3(T_j)+\vol_3(T_{j+1})\right)^2\\
    &=5^2(3!)^2\left\langle \begin{bmatrix}
    2&1&0&1\\
    1&2&1&0\\
    0&1&2&1\\
    1&0&1&2
\end{bmatrix}\begin{bmatrix}
  T_1\\T_2\\T_3\\T_4  
\end{bmatrix},\begin{bmatrix}
  T_1\\T_2\\T_3\\T_4  
\end{bmatrix} \right\rangle\\
&=\mathcal{Z}_{\text{pyr}}(Q).
\end{align*}
In this case, \eqref{T-D ineq} yields
\begin{equation}\label{eqn: pyramid edge inequality}
    \mathcal{Z}_{\text{pyr}}(Q)\leq \frac{\left(\|\vol_1(Q) \|_{\ell_2}^{2}+\|\vol_1(Q^o) \|_{\ell_2}^{2}\right)^3}{5^23^3(3!)^2},
\end{equation}
where equality can only be obtained when the corresponding frame is tight, which occurs for a square-based pyramid whose apex is the north pole and whose base lies in the plane $z=-\frac{3}{7}$.

We may once again rephrase this into norm inequalities.

\begin{theorem}\label{thm: pyramid ineq}
    Suppose that $Q\in\mathcal{P}_3$ is a pyramid and let $p$ be the intersection of the diagonals of the base, then
    \begin{equation}\label{thm: pyr ineq 1}
        \| \vol_{3}(Q)\|_{\ell_2,0}\leq\frac{1}{(3!)}\left(\frac{5}{3}\right)^{\frac{3}{2}}, \quad\text{and}
    \end{equation}
    \begin{equation}\label{thm: pyr ineq 2}
        \|\vol_3(Q) \|_{\ell_2,p}^2+\varepsilon_{\pyr}(Q)\leq \frac{\left(\|\vol_1(Q) \|_{\ell_2}^{2}+\|\vol_1(Q^o) \|_{\ell_2}^{2}\right)^3}{2\cdot5^23^3(3!)^2}, 
    \end{equation}
    where $\varepsilon_{\pyr}(Q)=\langle(\frac{1}{2}Z_{\pyr}-I)T,T\rangle$.
\end{theorem}

\begin{remark}
    The inequalities above become inequalities when the corresponding frames are tight.  This occurs in \eqref{thm: pyr ineq 1} for a square-based pyramid $Q\in\mathcal{P}_3$ whose apex is the north pole and whose base lies in the plane $z=\frac{1}{\sqrt{6}}$.  For \eqref{thm: pyr ineq 2}, this again occurs for a square-based pyramid $Q\in\mathcal{P}_3$ with apex the north pole and whose base lies in the plane $z=-\frac{3}{7}$.

    If we omit the edges that correspond to the diagonals of the base, we get a tight frame for a pyramid whose apex is the north pole and whose base is a square lying in the plane $z=-\frac{1}{5}$.  
\end{remark}
\begin{remark}
    Note that the volume optimizer in this class is a square-based pyramid whose apex is the north pole; however, the base lies in the plane $z=-\frac{1}{3}$. 
\end{remark}

%%%%%%%%%%%
%%%%%%%%%%% bipyramids
%%%%%%%%%%%

\subsection{Bipyramids}

For the bipyramid, we need to add one edge between $v_4$ and $v_5$ as seen in the figure below.
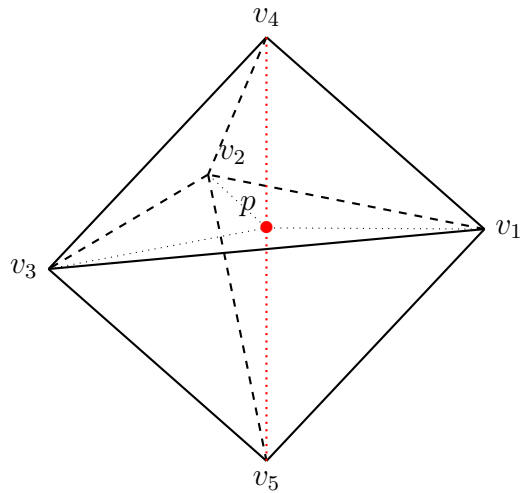
\begin{center}
 \tdplotsetmaincoords{70}{165}
 \begin{tikzpicture}[tdplot_main_coords,scale=3]
 
    \coordinate (O) at (0,0,0);

%    \draw[thick,->] (0,0,0) -- (1,0,0) node[anchor=north east]{$x$};
%    \draw[thick,->] (0,0,0) -- (0,1,0) node[anchor=north west]{$y$};
%    \draw[thick,->] (0,0,0) -- (0,0,1) node[anchor=south]{$z$};

    \tdplotsetcoord{A}{1}{90}{0}    % cartesian (1,0,0)
    \tdplotsetcoord{C}{1}{90}{180}  % cartesian (-1,0,0)
    \tdplotsetcoord{D}{1}{90}{270}  % cartesian (0,-1,0)
    \tdplotsetcoord{E}{1}{0}{0}     % cartesian (0,0,1)
    \tdplotsetcoord{F}{1}{180}{0}   % cartesian (0,0,-1)
        % the foreground edges
            \draw[thick] (A) -- (C) ;
            \draw[thick] (E) -- (A) ;
            \draw[thick] (A) -- (F) ;
            \draw[thick] (E) -- (C) ;
            \draw[thick] (C) -- (F) ;
        % the background edges
            \draw[dashed,thick] (C) -- (D) ;
            \draw[dashed, thick] (D) -- (A) ;
            \draw[dashed,thick] (E) -- (D) ;
            \draw[dashed,thick] (D) -- (F) ;
        % the internal edge
            \draw[dotted,red,thick] (E)--(F) ;
            \draw[dotted] (A) -- ($.5*(E)+.5*(F)+(0,0,.1)$);
            \draw[dotted] (D) -- ($.5*(E)+.5*(F)+(0,0,.1)$);
            \draw[dotted] (C) -- ($.5*(E)+.5*(F)+(0,0,.1)$);
        %center point
            \node[red] at ($.5*(E)+.5*(F)+(0,0,.1)$) {\textbullet};
        % labels       
            \node[left] at (A) {\(v_3\)};
            \node[right] at (C) {\(v_1\)};
            \node[above right] at (D) {\(v_2\)};
            \node[above] at (E) {\(v_4\)};
            \node[below] at (F) {\(v_5\)};
            \node[left] at ($.5*(E)+.5*(F)+(0,0,.2)$) {$p$};
  \end{tikzpicture}
\end{center}

Using a similar notation to that found in the pyramid case, we denote by $V_j$ the matrix consisting of the edge vectors incident to $v_j$.  The volume may be partitioned into tetrahedra $\{T_j:1\leq j\leq 6\}$, such that $e_j\in\mathcal{F}_1(T_j)\cap\mathcal{F}_1(T_{j+3})$ and $v_4\in T_j$ for $j=1,2,3$.  We have
\begin{align*}
    \det(S_E(Q\cup Q^o)) &= 25\left(\sum_{j=1}^{5}\det(V_jV_j^T)^2 \right)\\
    &=5^2(3!)^2\sum_{i=0}^{1}(\vol_3(T_{3i+1})+\vol_3(T_{3i+2})+\vol_3(T_{3i+3}))^2\\
    &\qquad +5^2(3!)^2\sum_{j=1}^{3}(\vol_3(T_j)+\vol_3(T_{j+3}))^2 \\
    &=5^2(3!)^2\left\langle \begin{bmatrix}
    2 & 1& 1&1&0&0\\
    1&2&1&0&1&0\\
    1&1&2&0&0&1\\
    1&0&0&2&1&1\\
    0&1&0&1&2&1\\
    0&0&1&1&1&2
\end{bmatrix}\begin{bmatrix}
    T_1\\T_2\\T_3\\T_4\\T_5\\T_6
\end{bmatrix}, \begin{bmatrix}
    T_1\\T_2\\T_3\\T_4\\T_5\\T_6
\end{bmatrix} \right\rangle\\
    &=:5^2(3!)^2\mathcal{Z}_{\text{bipyr}}(Q)
\end{align*}

\begin{theorem}
Suppose that $Q\in\mathcal{P}_3$ is a bipyramid labeled as above.  For the point $p=\conv(v_1,v_2,v_3)\cap \conv(v_4,v_5)\in\Int(Q)$, we have
\begin{equation}\label{thm: bipyr ineq 1}
 \|\vol_{3}(Q)\|_{\ell_{2},0}^2 \leq \dfrac{1}{3!^2}\left(\dfrac{5}{3} \right)^3, \text{ and}
\end{equation}
\begin{equation}\label{thm: bipyr ineq 2}
\|\vol_3(Q)\|_{\ell_2,p}^2+\varepsilon_{\bipyr}(Q) \leq \frac{\left(\|\vol_1(Q)\|_{\ell_{2}}^{2}+\|\vol_{1}(Q^o)\|_{\ell_2}^2\right)^3}{2\cdot5^23^3(3!)^2},
\end{equation}
where $\varepsilon_{\bipyr}(Q) =\langle(\frac{1}{2}Z_{\bipyr}-I)T,T \rangle.$
\end{theorem}

\begin{remark}
   In $\mathcal{P}_3$, inequalities \eqref{thm: bipyr ineq 1} and \eqref{thm: bipyr ineq 2} are not sharp.  However, if we relax the condition to $Q\subset B_2(0,1)$, then we can generate tight frames, which produce the equality cases.  For either inequality to be an equality we need $\conv(v_1,v_2,v_3)$ to be a regular triangle and we need $\conv(v_4,v_5)$ to be orthogonal to the plane containing the triangle and also pass through its centroid.  Let $R$ denote the circumradius of the triangle $\conv(v_1,v_2,v_3)$ and $2h$ denote the length of $\conv(v_4,v_5)$, then both \eqref{thm: bipyr ineq 1} and \eqref{thm: bipyr ineq 2} are equalities whenever $\frac{h}{R}=\frac{\sqrt{3}}{2}$.  
   
   In this case, we have $S_E(Q)=(\frac{3}{4}R^2)I$.  Since $R\leq 1$, the largest such constant corresponds to $\conv(v_1,v_2,v_3)$ being inscribed in a great circle of $\mathbb{S}^2$.
\end{remark}

\begin{remark}
    If we left out the internal edge between $v_4$ and $v_5$ above, we could generate a tight frame when $\frac{R}{h}=\frac{2}{\sqrt{5}}$.  In this case, we have $S'_{E}(Q)=(6h^2)I$ and the largest such possibility has $R=\frac{2}{\sqrt{5}}$.
\end{remark}

%%%%%%%%%%%
%%%%%%%%%%% linear algebra stuff about the quadratic form
%%%%%%%%%%%

Choose a point $p\in\Int(Q)$, we partition into simplices $(T_j:1\leq j\leq n_d)$ bounded by segments of the edges of $Q$ such that
\begin{align*}
(a)&\quad Q=\bigcup_{j=1}^{n_d}T_j,\\ (b)&\quad\displaystyle\bigcap_{j=1}^{n_d} \mathcal{F}_0(T_j)=\{p\},\text{ and }\\
(c)&\quad \Int(T_i)\cap\Int(T_j)=\emptyset;\quad i\neq j.
\end{align*}

\begin{remark}
    Since $Q$ is convex, a natural choice is $p=C(Q)\in\Int(Q)$, the centroid of $Q$.
\end{remark}

Denote by $T$ the vector of $d$-volumes, $T=\begin{bmatrix} \vol_d(T_1)&\cdots&\vol_d(T_{n_d}) \end{bmatrix}$. Then the edge frame provides an inequality of the form
\[
c_d\mathcal{Z}(Q) \leq \frac{\left( \|\vol_1(Q) \|_{\ell_2}^2+\|\vol_1(Q^o) \|_{\ell_2}^2\right)^d}{d^d}, 
\]
where $\mathcal{Z}(Q)=\langle ZT,T\rangle$ is a quadratic form that enjoys the following easily checked properties.
\begin{enumerate}[(a)]
    \item The $n_d\times n_d$ matrix representation of $\mathcal{Z}$, denoted by $Z=[z_{i,j}]$, is symmetric \[ z_{i,j}=z_{j,i};\quad 1\leq i,j\leq n_d.\]
    \item The matrix $Z$ is positive semi-definite on $\mathbb{R}_{+}^{n_d}$ \[ \left\langle Zx,x \right\rangle\geq 0;\quad x\in\mathbb{R}_{+}^{n_d}.\]
   % {\color{blue} This follows from
    %\[
   % \langle Z T,T\rangle = (d+2)^{d-1}\sum\left|v_{j_1}\cdots v_{j_d}\right|^2 =\det(S_E(Q))\geq 0.
   % \]}
    \item The entries of $Z$ are non-negative \[ z_{i,j}\geq 0;\quad 1\leq i,j\leq n_d.\]
    %{\color{blue} Not too much to say about this one.  It follows from expanding all the squares. }
    \item $Z$ is constant along the diagonal \[ z_{i,i}=\alpha;\quad 1\leq i\leq n_d.\]
    %{\color{blue} Since the over count factor is consistent, we get a constant diagonal. }
    \item The sum of each row is constant \[  \sum_{j=1}^{n_d}z_{i,j}=\varrho; \quad 1\leq i\leq n_d.\]
    %{\color{blue} This follows from the fact that the number of cells adjacent to each cell $T_j$ is the same.  $T_j$ is adjacent $T_i$ if they share a facet, that is $\mathcal{F}_{d-1}(T_i)\cap\mathcal{F}_{d-1}( T_j)\neq 0$. }
\end{enumerate}
The following result is a straightforward of consequence of the above properties.
\begin{lemma}\label{lem: spectral radius}
Suppose that $Z$ satisfies the properties above, and denote by $\mathcal{E}(Z)$, the set of eigenvalues for $Z$.  Then
\begin{enumerate}[(i)]
    \item $\mathcal{E}(Z)\subset[0,\varrho]$,
    \item $\max(\mathcal{E}(Z))=\varrho$, and
    \item $Z \mathbf{1}_{n_d}=\varrho\mathbf{1}_{n_d},$ where $\mathbf{1}=\begin{bmatrix}
        1&\cdots&1
    \end{bmatrix}^{T}$.
    \item For all $\lambda\in\mathcal{E}(Z)$ with $\lambda\neq\varrho$, the eigenvector $v_\lambda=\begin{bmatrix}
        v_{\lambda,1} & \cdots & v_{\lambda,n_d}
    \end{bmatrix}^T$ satisfies\\ $\min(\{v_{\lambda,j}:1\leq j\leq n_d\})<0$, that is, $v_\lambda \notin \mathbb{R}^{n_d}_{+}$.  
\end{enumerate}
\end{lemma}

\begin{remark}
    As a consequence of $(iv)$, the only eigenvector that corresponds to a realizable polytope is $v_\sigma=\mathbf{1}_{n_d}$.  
\end{remark}

%%%%%%%%%%%%%%%%%%
%%%%%%%%%%%%%%%%%%
%%%%%%%%%%%%%%%%%%

\begin{comment}
\begin{remark}
    For $Q\in\mathcal{P}_2$, the corresponding degeneracy with respect to perimeter $\mathcal{E}_1(Q)$ is given by
    \[
    \mathcal{E}_1(Q):=\left(\vol_1(e_1)+\vol_1(e_3)\right)\left(\vol_1(e_2)+\vol_1(e_4) \right)
    \]
    which satisfies
    \[
    0\leq \mathcal{E}_1(Q)\leq \left(\frac{\vol_1(Q)}{2}\right)^2.
    \]
    The equality conditions are similar to those for $\mathcal{E}_2(Q)$.  The upper bound is obtained by a square.
\end{remark}

\end{comment}

\begin{conjecture}
    Let $d\in\mathbb{N}$ and suppose $Q\subset\mathbb{S}^{d-1}$ is a convex polytope with $d+2$ vertices.  Then the following holds
    \[
    \|\vol_d(Q) \|_{\ell_2}^2+\varepsilon_d(Q)\leq \frac{\left(\| \vol_1(Q)\|_{\ell_2}^2+\|\vol_d(Q^o) \|_{\ell_2}^{2}\right)^d}{ c_d\cdot d!^2\cdot d^d},
    \]
    where $c_d=2(d+2)^{d-1}$ and $\varepsilon_d$ is a quadratic form that satisfies
    \[
    0\leq \varepsilon_d(Q)\leq C\vol_d(Q)^2.
    \]
\end{conjecture}
\begin{remark}
    A collection of $d$ edge vectors (actual or synthetic) will bound a closed simplicial region with $d$-content in $Q$.  We partition $Q$ into these regions and take the center to be the point of intersection of these closed regions.  The number of regions depends on the configuration of the edges as can be seen already in $\mathbb{R}^3$ above.
\end{remark}
\begin{remark}
    The constant $c_d=2(d+2)^{d-1}$ above is the number of spanning trees on the complete graph $K_{d+2}$ where we have removed an edge, see \cite[A007334]{oeis}.
\end{remark}
\begin{remark}
    In general, it is not possible to generate a sharp bound for the inequality above.  This limitation exists because it is not always possible to generate a tight frame for a vector space with vectors that satisfy desired constraints.  See for instance, \cite{fickus_2016} and \cite{SUSTIK_2007}.  However, similar to the bipyramid above, we may be able to find tight frames if we allow some of the vertices to be on the interior of the unit ball.
\end{remark}

\bibliographystyle{plainurl} 
\bibliography{main}

\end{document}